\documentclass[reqno]{amsart}

\usepackage{comment}
\usepackage{color,cite}
\usepackage{ulem}
\usepackage{hyperref}                   
\hypersetup{colorlinks,
    linkcolor=blue,%
    citecolor=blue}

\usepackage{amssymb, longtable,mathrsfs, comment}
\usepackage{amsmath,amsfonts,amssymb,amsthm,amscd,latexsym}

\usepackage{color}
\usepackage{cite}
\usepackage{graphicx}

\usepackage{hyperref}
\usepackage{enumitem}
\usepackage{ulem}
\hypersetup{colorlinks,
	linkcolor=blue,%
	citecolor=blue}	

\newtheorem{theorem}{Theorem}[section]
\newtheorem{lemma}[theorem]{Lemma}

\newtheorem{prop}[theorem]{Proposition}
\newtheorem{corollary}[theorem]{Corollary}

\newtheorem{remark}[theorem]{Remark}

\newtheorem{prob}[theorem]{Problem}

\numberwithin{equation}{section}
\newcommand{\norm}[1]{\left\lVert#1\right\rVert}

\newcommand{\cM}{{\mathcal M}}
\newcommand{\cN}{{\mathcal N}}
\newcommand{\cA}{{\mathcal A}}

\newcommand{\cP}{{\mathcal P}}

\newcommand{\cL}{{\mathcal L}}
\newcommand{\cF}{{\mathcal F}}

\newcommand{\cU}{{\mathcal U}}

\newcommand{\abs}[1]{\left|#1\right|}

\begin{document}

\author[J. Huang]{Jinghao Huang}
\address{Institute for Advances Study in Mathematics, Harbin Institute of Technologies,
Harbin, 150001, China}
\email{jinghao.huang@hit.edu.cn}

\author[K. Kudaybergenov]{Karimbergen Kudaybergenov}
\address{Institute for Advances Study in Mathematics, Harbin Institute of Technologies,
Harbin, 150001, China and
Suzhou Research Institute of Harbin Institute of Technology, Suzhou, 215104, China}
\email{kudaybergenovkk@gmail.com}

\author[F. Sukochev]{Fedor Sukochev}
\address{School of Mathematics and Statistics, University of New South Wales,
Kensington, 2052, Australia}
\email{f.sukochev@unsw.edu.au}

\author[B. Yan]{Bing Yan$^{\ast}$}\thanks{*The corresponding author.}
\address{Institute for Advances Study in Mathematics, Harbin Institute of Technologies,
Harbin, 150001, China}
\email{bingyan202510@163.com}

\title[Isometries of Haagerup--Schultz algebras]{Isometries of  Haagerup--Schultz  algebras}

\begin{abstract}
We establish a version of 
the classical results concerning descriptions of isometries on $C^*$-algebras and noncommutative $L_p$-spaces due to Kadison (1951) and Yeadon (1981)  in the setting of Haagerup--Schultz algebras. 
Precisely, we show that (not necessarily surjective) isometries on such algebras 
are necessarily implemented by partial isometries and trace-preserving Jordan $^*$-monomorphisms.


\end{abstract}


\maketitle

\bigskip

\section{Introduction}

An \textit{isometry} between operator algebras (such as $C^*$-algebras or von Neumann algebras) is a (not necessarily surjective) linear map $\Phi:\mathcal{A}\to\mathcal{B}$ satisfying
\begin{equation*}
\left\|{\Phi(x)}\right\|_{\mathcal{B}}=\left\|{x}\right\|_{\mathcal{A}}\quad \text{for all } x\in\mathcal{A}
\end{equation*}
A milestone   in this area is due to Kadison~\cite{K51}, who established a noncommutative version of the classical Banach--Stone Theorem: 
\begin{quote}
\textbf{Kadison's Theorem (1951)} \cite{K51}.
For a linear bijection $\Phi:\mathcal{A}\to\mathcal{B}$ between unital $C^*$-algebras,
\begin{align}\label{descriptionofKadison}
\Phi \text{ is an isometry } \iff \Phi(x)=u\,J(x),\; x\in\mathcal{A},
\end{align}
where $u\in\mathcal{B}$ is unitary ($u^*u=uu^*=\mathbf{1}$) and $J:\mathcal{A}\to\mathcal{B}$ is a Jordan $^*$-isomorphism (preserving the Jordan product $x\circ y=\frac{1}{2}(xy+yx)$ and the involution).
\end{quote}
Russo and Dye\cite{RD66} established Kadison's theorem in a more general framework by showing that 
a linear map between unital $C^*$-algebras that sends unitaries to unitaries is a composition of multiplication by a unitary and a Jordan $^*$-homomorphism.

In the $1930$s,  Banach  obtained the general form of isometries between $L_p$-spaces on a finite measure space\cite{Banach}, which was later extended by Lamperti to a more general setting\cite{Lamperti}. 
After the non-commutative $L_p$-spaces were introduced by Dixmier \cite{Dixmier} and Segal \cite{Segal} in the 1950s, the study of $L_p$-isometries was conducted by Broise \cite{Broise}, Russo\cite{Russo68}, Arazy \cite{Arazy}, Tam \cite{Tam}, etc. A complete description (for the semifinite case) was obtained in 1981 by Yeadon \cite{Yeadon80} (see also  \cite{Sherman05} for the case when  $0<p\le1$). 
A complete description for the limiting case, isometries on noncommutative $L_0$-spaces, is given in a recent paper\cite{BHS24}, 
which extended the Banach--Stone theorem and Kadison's theorem for isometries of von Neumann algebras (see also \cite{HKS} for $L_0$-isometries with respect to another metric). 
For recent progress on the description of isometries on   noncommutative symmetric spaces, we refer the reader to  \cite{HSZ20, HS24} and references therein.

In 1952, Fuglede and Kadison~\cite{FK52} extended the classical notion of the determinant to the setting of  type~II$_1$-factors, 
which has become a fundamental  concept in the theory of operator algebras. Subsequently, Haagerup and Schultz~\cite{HS07} generalized this notion to the setting of 
arbitrary von Neumann algebras equipped with a faithful normal finite trace and  introduced the concept of noncommutative $\mathcal{L}_{\log}$-spaces   of log-integrable operators. Later, Dykema, Sukochev, and Zanin~\cite{DSZ16, DSZ17a} defined an $F$-norm $\left\|{\cdot}\right\|_{\log}$ on the space  \(\cL_{\log}(\cM, \tau)\) and showed
that this space is 
a complete topological $^\ast$-algebra.
 The space \(\cL_{\log}(\cM, \tau)\)  is called the \textit{Haagerup--Schultz algebra} in the special case when $\cM$ is  a group von Neumann algebra\cite{LT14}, where 
 the authors provided an  elegant way to define  \(L^2\)-torsion via  the Haagerup--Schultz algebra\cite[Section~6]{LT14}.
In the present paper,  for a given von Neumann algebra $\cM$ equipped with a semifinite faithful normal trace,  we call $\cL_{\log}(\cM,\tau)$ the Haagerup--Schultz algebra associated with $\cM$.

While the Banach--Stone Theorem, the
Banach--Lamperti Theorem and their noncommutative generalizations provide descriptions of  isometries of 
$C^*$-algebras and commutative/noncommutative 
$L^p$-spaces, the description of  isometries on  Haagerup--Schultz algebras is unexplored. 
The $F$-norm $\left\|{\cdot}\right\|_{\log}$  is non-homogeneous,
which causes several difficulties in the study of isometries on $\cL_{\log}(\cM,\tau)$.
Recently, it is shown in \cite{ACM20} that two spaces
\[
\big(\mathcal{L}_{\log}(\Omega_1,\mu_1),\left\|{\cdot}\right\|_{\log}\big)
\quad\text{and}\quad
\big(\mathcal{L}_{\log}(\Omega_2,\mu_2),\left\|{\cdot}\right\|_{\log}\big)
\]
over two measure spaces respectively
are isometric if and only if there exists a measure-preserving isomorphism between $(\Omega_1,\mu_1)$ and $(\Omega_2,\mu_2)$.

The main goal of the present paper is to give a full description of (not necessarily surjective) isometries on 
the Haagerup--Schultz algebra
$\cL_{\log}(\cM,\tau)$, which complements and generalizes previous results obtained in \cite{K51}, \cite{Yeadon80} and \cite{ACM20}.
Most of existing literature concerning descriptions of isometries on Banach spaces only treat surjective isometries, see   Kadison\cite{K51} (and \cite{FJ08,FJ03,ACM20,BHS24,HS24} and references therein).
Moreover, it is given in \cite[Example 4.7]{BHS24}  that an example of an injective isometry on an $L_\infty$-space/$L_0$-space which does not admit an elementary form similar to  
  \eqref{descriptionofKadison}, i.e., it does not preserve disjointness. 
Until now, commutative/noncommutative $L_p$-spaces\cite{Yeadon80} and  the Hardy space $H^p$\cite{FJ03,Forelli}, $0<p<\infty$,   are   the only known examples of   $F$-spaces admitting  elementary descriptions of isometries  that are not necessarily surjective.
In the present paper, 
we do not require the isometries to be surjective.
This  provides a new example of an $F$-space, on which all  (not necessarily  surjective) isometries are of  elementary forms. 

\begin{theorem}\label{log-isometry}
Let $\mathcal{M}$ and $\mathcal{N}$ be von Neumann algebras equipped with faithful normal semifinite traces $\tau$ and $\nu$, respectively, and let $\Phi:\mathcal{L}_{\log}(\mathcal{M}, \tau) \to \mathcal{L}_{\log}(\mathcal{N}, \nu)$ be a linear mapping (not necessarily surjective). Then the following are equivalent:
\begin{enumerate}
    \item $\Phi$ is an isometry from $\left(\mathcal{L}_{\log}(\mathcal{M}, \tau), \left\|\cdot\right\|_{\log}\right)$ into $\left(\mathcal{L}_{\log}(\mathcal{N}, \nu), \left\|\cdot\right\|_{\log}\right)$;
    \item $\Phi$ admits the decomposition
    \begin{align}\label{gen-form}
        \Phi(x) = w J(x), \quad x \in \mathcal{L}_{\log}(\mathcal{M}, \tau),
    \end{align}
    where  $J:\cM \to \cN$  is a normal Jordan $^*$-monomorphism, the restriction $J|_{\cM\cap \cL_{\log}(\cM, \tau)}$ extends as a Jordan $^*$-monomorphism from $\mathcal{L}_{\log}(\mathcal{M}, \tau)$ into $\mathcal{L}_{\log}(\mathcal{N}, \nu)$ satisfying $\nu(J(x)) = \tau(x)$   for all $0\le x \in \cM$ and
    $w$ is a partial isometry in $\mathcal{N}$ such that $w^*w=J(\mathbf{1})$.
\end{enumerate}
Moreover, if $\Phi$ is surjective, then $w$ is a unitary and $J$ is a Jordan $^*$-isomorphism from $\cM$ onto $\cN$.
\end{theorem}

This paper is structured as follows.
In Section~\ref{HSa}, we recall necessary notions from  the  noncommutative integration theory, 
including algebras of measurable operators, Haagerup--Schultz algebras and trace-preserving Jordan $^*$-monomorphisms.
Section~\ref{IsoHSa} consists of three subsections. In Subsection~\ref{subsec-1},
we give an expression of the  $F$-norm $\left\|\cdot\right\|_{\log}$ on 
the Haagerup--Schultz algebra $\cL_{\log}(\cM, \tau)$ in terms of  distribution functions  (see Lemma~\ref{key-lm1-1}).
Having this expression at hand,
we show that $\cL_{\log}$-isometry preserves singular value functions of  operators from the noncommutative $L_1$-space, which is the first
key step of the proof of the main result, where
the
so-called Stieltjes transforms plays an important role.
In Subsection~\ref{sin-val-pre}, we show that an isometry of a Haagerup--Schultz algebra sends partial isometries with finite-trace support into partial isometries.
 Subsection~\ref{proo$F$-o$F$-m-th} provides the proof of the   main result of the present paper (Theorem~\ref{log-isometry}). 
In~Section~\ref{S-l_p}, we examine the relationship between isometries of Haagerup--Schultz algebras and those of non-commutative $L_p$-spaces associated with semifinite von Neumann algebras.
In~Section~\ref{S-coincide}, we provide a complete characterization of those semifinite von Neumann algebras $(\mathcal M,\tau)$ for which the Haagerup--Schultz algebra $\mathcal L_{\log}(\mathcal M,\tau)$ coincides with the space $L_1(\mathcal M,\tau)$.
We also characterise the one-parameter group of surjective isometries on $\cL_{\log}(H)$, which  is the first characterization of such a group beyond Banach spaces.

\section{Preliminaries}\label{HSa}

\subsection{$\tau$-Measurable operators}

Let $H$ be a complex Hilbert space and let $B(H)$ denote the $^\ast$-algebra of all bounded linear operators on $H$. Let $\cM$ be a von Neumann algebra contained in $B(H)$. As usual, we denote by $P(\cM)$ the set of all projections in $\cM$.


Recall that a densely defined closed linear operator $x : \mathrm{dom}(x) \to H$ (where $\mathrm{dom}(x)$ is a dense linear subspace of $H$) is said to be \textit{affiliated} with $\cM$ if
\[
yx \subset xy \quad \text{for all } y \in \cM',
\]
where $\cM'$ denotes the commutant of $\cM$ (see \cite{BCLSZ,DPS}).
An operator $x$ affiliated with $\cM$ is called \textit{measurable} (with respect to $\cM$) if
\[
e_{(\lambda,\infty)}(|x|) \text{ is a finite projection for some } \lambda > 0,
\]
where $e_{(\lambda,\infty)}(|x|)$ is the spectral projection of $|x|$ corresponding to the interval $(\lambda, \infty)$.
Let $\tau$ be a faithful normal semifinite trace on $\cM$. A measurable  operator $x$ affiliated with $\cM$ is called \textit{$\tau$-measurable} if
\[
\tau\big(e_{(\lambda,\infty)}(|x|)\big) < \infty \text{ for some } \lambda > 0.
\]
We denote by $S(\cM)$ and $S(\cM, \tau)$ the sets of all measurable and $\tau$-measurable operators, respectively  (see \cite{BCLSZ,DPS}).

For $x, y \in S(\cM)$, it is well known that $x+y$ and $xy$ are densely defined and preclosed operators. Moreover, their closures, as well as $x^\ast$, belong to $S(\cM)$. 
Equipped with these operations, both $S(\cM)$ and $S(\cM, \tau)$ form unital $^\ast$-algebras over $\mathbb{C}$ (see \cite{BCLSZ,DPS}). 
Clearly, $\cM$ itself is a $^\ast$-subalgebra of $S(\cM, \tau)$ and  $S(\cM, \tau)$ is a $^\ast$-subalgebra of $S(\cM)$.

For an arbitrary $x \in S(\cM)$, we define:
\begin{itemize}
    \item \textit{left support} $l(x)$: the smallest projection $p \in \cM$ such that $px = x$;
    \item \textit{right support} $r(x)$: the smallest projection $q \in \cM$ such that $xq = x$;
    \item \textit{support projection} $s(x) = l(x) \vee r(x)$.
\end{itemize}
Note that $s(|x|) = l(|x|) = r(|x|)$ for all $x \in S(\cM)$.

Consider the topology $t_\tau$ of convergence in measure (also called the \textit{measure topology}) on $S(\cM, \tau)$ (see \cite{BCLSZ,DPS}). A neighborhood basis of zero is given by
\[
N(\varepsilon, \delta) = \Big\{x\in S(\cM, \tau) : \exists\, e \in P(\cM), \; \tau(\mathbf{1} - e) \le \delta, \; xe \in \cM, \; \left\|{xe}\right\|_{\cM} \le \varepsilon \Big\},
\]
where $\varepsilon, \delta > 0$. The pair $(S(\cM, \tau), t_\tau)$ is a complete topological $^\ast$-algebra.

\subsection{Generalized singular value functions and distribution functions}

The distribution function of  positive measurable function $x$ on $(0;\infty)$  defined by
\begin{align*}
d(s; x)  = m\left\{t > 0 : x(t) >  s \right\} = \int\limits_0^\infty \chi_{(s;\infty)}(x(t)) \, dt, \quad s \ge 0,
\end{align*}
where $m$ denotes the Lebesgue measure on $\mathbb{R}$.

For $x \in S(\cM, \tau)$, the \textit{spectral distribution function} $d(|x|) = d(\cdot; |x|)$ of $|x|$ is defined by
\[
d(s; |x|) = \tau\big(e_{(s, \infty)}(|x|)\big), \quad s \ge 0
\]
(see \cite[Chapter 4, Section 4.1]{BCLSZ} and \cite[Chapter 3, Section 3.2]{DPS}).

Note that
\begin{align}\label{dslam}
d(s; \lambda |x|) & = d\left(\frac{s}{\lambda}; |x|\right)
\end{align}
for all $s\ge 0$ and $\lambda>0$ (see e.g. \cite[Theorem 1.5.14]{DPS}). 

The \textit{generalized singular value function} $\mu(x)$ of $x \in S(\cM, \tau)$ is given by
\begin{align}\label{mu(x)}
\mu(t; x) &= \inf\{ s > 0 : d(s; |x|) \le t \} \\
&= \inf\big\{ \left\| x (\mathbf{1} - p) \right\|_{\cM} : p \in P(\cM), \; \tau(p) \le t \big\}. \nonumber
\end{align}
Moreover (see \cite[p.130]{DPS}),
\begin{align}\label{dpdmu}
d(s; |x|) = m\big\{ t \ge 0 : \mu(t; x) > s \big\} = d(s; \mu(x)), \quad s \ge 0.
\end{align}

The following lemma summarizes several  properties of generalized  singular value functions (see \cite[Proposition 4.1.9 and Example 4.1.5]{BCLSZ}).

\begin{lemma}\label{mu-prop}
If $x \in S(\cM, \tau)$, then:
\begin{itemize}
    \item[(a)] $\mu(t; \alpha x) = |\alpha| \mu(t; x)$ for all $\alpha \in \mathbb{C}$ and $t > 0$;
    \item[(b)] if $0 \le x \le y$, then $\mu(t; x) \le \mu(t; y)$ for all $t > 0$;
    \item[(c)] $\mu(x) = \chi_{[0, r)}$ if and only if $|x|$ is a projection with $\tau(|x|) = r$;
    \item[(d)] $\mu(t; x) \to \left\|x\right\|_{\cM}$ as $t \to 0$ for bounded $x$, and $\mu(t; x) \to \infty$ as $t \to 0$ for unbounded $x$.
\end{itemize}
\end{lemma}

The \textit{noncommutative $L_1$-space} $L_1(\cM, \tau)$ is defined as
\begin{align*}
L_1(\cM, \tau) = \left\{ x \in S(\cM, \tau) : \int\limits_0^\infty \mu(t; x) \, dt < \infty \right\},
\end{align*}
with the norm
\[
\left\|x\right\|_1 = \int\limits_0^\infty \mu(t; x) \, dt, \quad x \in L_1(\cM, \tau).
\]

By \cite[Propositions 1.1.4 and 1.4.5]{Gra24}, for any positive measurable function $\phi$ on $(0,\infty)$ we have
\begin{align}\label{prop18}
\int\limits_0^\infty \phi(t)\, dt=\int\limits_0^\infty \mu(t;\phi)\, dt=\int\limits_0^\infty d(t;\phi)\, dt.
\end{align}
Hence, for any $x\in L_1(\cM, \tau)$, we have
\begin{align}\label{iden}
\left\|x\right\|_1 &: = \tau(|x|) = \int\limits_0^\infty \mu(t;x)\, dt\stackrel{\eqref{prop18}}{=} \int\limits_0^\infty d(t;\mu(x))\, dt \stackrel{\eqref{dpdmu}}{=} \int\limits_0^\infty d(t;|x|)\, dt.
\end{align}

Let us recall the following useful property.  
\begin{prop}\label{phi-f} \cite[Proposition 1.1.4]{Gra24}
For any increasing continuously differentiable function $\phi$ on $[0,\infty)$ with
$\phi(0) = 0$ and every measurable function $f$ on $[0,\infty)$ with $\phi(|f|)$ integrable on $[0,\infty)$, we
have
\begin{align}\label{phi|f|}
\int\limits_0^\infty \phi(|f(t)|)\,dt & = \int\limits_0^\infty \phi'(s)d(s; |f|)\, ds.
\end{align}
\end{prop}

\subsection{Haagerup--Schultz algebras}

Denote by $\cL_{\log}(\cM, \tau)$ the set of all elements $x \in S(\cM, \tau)$ such that \cite{DSZ16}
\[
\tau\bigl( \log(\mathbf{1} + \abs{x}) \bigr) < \infty.
\]
Equipped with the $F$-norm
\begin{align}\label{log-norm}
\left\|x\right\|_{\log} = \tau\bigl( \log(\mathbf{1} + \abs{x}) \bigr) = \int\limits_0^\infty \log\bigl(1+\mu(t;x)\bigr)\,dt, \qquad x \in \cL_{\log}(\cM, \tau),
\end{align}
the space $\cL_{\log}(\cM, \tau)$ forms a complete topological $^\ast$-algebra~\cite{DSZ16}  (see also \cite[Proposition 3.1]{DDSZ20}). The $F$-norm $\left\|\cdot \right\|_{\log}$ satisfies:
\begin{enumerate}
    \item $\left\|{x}\right\|_{\log} > 0$ for $0 \ne x \in \cL _{\log}(\cM,\tau)$;
    \item $\left\|{x^*}\right\|_{\log} = \left\|x\right\|_{\log}$ for $x \in \cL _{\log}(\cM,\tau)$;
    \item $\left\|\lambda x \right\|_{\log} \leq \left\|x\right\|_{\log}$ if $|\lambda| \leq 1$ and  $x \in \cL _{\log}(\cM,\tau)$;
    \item $\lim\limits_{\lambda \to 0} \left\|\lambda x\right\|_{\log} = 0$  for $x \in \cL _{\log}(\cM,\tau)$;
    \item $\left\|x + y\right\|_{\log} \leq \left\|x\right\|_{\log} + \left\|y\right\|_{\log}$  for $x ,y \in \cL _{\log}(\cM,\tau)$.
\end{enumerate}
Furthermore, $L_1(\cM, \tau)$ is dense in $\cL_{\log}(\cM, \tau)$~\cite[Proposition~4.7]{DSZ16}.
If $\cM$ is a finite von Neumann algebra equipped with a faithful normal finite trace $\tau$, then $L_p(\cM,\tau)\subset L_{\log}(\cM,\tau)$ for all $p>0$.

The algebra $\cL_{\log}(\cM, \tau)$ is unital (that is, $\mathbf{1} \in \cL_{\log}(\cM, \tau)$) if and only if $\tau$ is finite. In this case, $\cM \subseteq \cL_{\log}(\cM, \tau)$ and is $\left\|\cdot\right\|_{\log}$-dense in $\cL_{\log}(\cM, \tau)$~\cite[Proposition~4.7]{DSZ16}.


\begin{remark}
Let $\cA$ be an abelian von Neumann algebra with $\cA \cong L_\infty(\Omega, \mu)$ for a measure space $(\Omega, \mu)$, and let $\tau$ be given by integration with respect to  $\mu$:
\[
\tau(x) = \int\limits_\Omega x(\omega)\,d\mu(\omega), \qquad x \in L_\infty(\Omega, \mu).
\]
Let $S(\Omega, \mu)$ denote the set of measurable complex-valued functions on $\Omega$, with the usual identification of functions that agree almost everywhere. Then abelian Haagerup--Schultz algebra and $F$-norm can be defined as follows 
\[
\cL_{\log}(\Omega, \mu) = \left\{ f \in S(\Omega,\mu): \int\limits_\Omega \log(1 + |f(\omega)|)\, d\mu(\omega) < \infty \right\}
\]
and
\[
\left\|f\right\|_{\log} = \int\limits_\Omega \log(1 + |f(\omega)|)\, d\mu(\omega),\, f \in \cL_{\log}(\Omega, \mu).
\]
\end{remark}

\subsection{Trace-preserving Jordan $^*$-monomorphisms induce isometries on Haagerup--Schultz algebras} \label{Jor}

Let $\mathcal{A}$ and $\mathcal{B}$ be associative $^*$-algebras. The \emph{Jordan product} on $\mathcal{A}$ is defined by
\begin{align*}
x \circ y = \frac{1}{2}(xy + yx), \qquad x,y \in \mathcal{A}.
\end{align*}

A linear mapping $\Psi: \mathcal{A} \to \mathcal{B}$ is called:
\begin{itemize}
    \item a \emph{homomorphism} if $\Psi(xy) = \Psi(x)\Psi(y)$ for all $x,y \in \mathcal{A}$;
    \item an \emph{anti-homomorphism} if $\Psi(xy) = \Psi(y)\Psi(x)$ for all $x,y \in \mathcal{A}$;
    \item a \emph{Jordan homomorphism} if
    \begin{align*}
    \Psi(x \circ y) = \Psi(x) \circ \Psi(y) \quad \text{for all } x,y \in \mathcal{A},
    \end{align*}
    or equivalently, $\Psi(x^2) = \Psi(x)^2$ for all $x \in \mathcal{A}$.
\end{itemize}

An 
injective (respectively, bijective)
homomorphism is called a \emph{monomorphism} (respectively, \emph{isomorphism}). Clearly, every (anti-)homomorphism   is a Jordan homomorphism.
If, in addition, $\Psi(x^*) = \Psi(x)^*$ for all $x \in \mathcal{A}$, then $\Psi$ is called a $^*$-homomorphism, $^*$-anti-homomorphism, and Jordan $^*$-homomorphism, respectively.

Let $(\cM,\tau)$ and $(\cN,\nu)$ be two semifinite von Neumann algebras. 
A Jordan $^*$-homomorphism $J$ from $S(\cM, \tau)$ into $S(\cN, \nu)$ is said to be normal  
if for any positive increasing net $\{x_i\}$ in $S(\cM,\tau)$, we have
\begin{align*}
\sup\limits_{i}J(x_i) & =J(\sup\limits_{i}x_i).
\end{align*}

Let  $J$ be  a Jordan $^*$-monomorphism from  $\cM$ into $\cN$ with $\nu(J(x))=\tau(x)$ for all $0\le x\in \cM$.
By \cite[Proposition 3.4]{BHS24}, $J$ extends as a Jordan $^*$-monomorphism from $S(\cM, \tau)$ into $S(\cN, \nu)$.
By \cite[Theorem 3.5]{BHS24}, $J$ preserves the singular value functions,
that is,
\begin{align}\label{pre-sin}
\mu(x) & =\mu(J(x)),\,\, x \in S(\cM, \tau).
\end{align}
For each $x\in \cL_{\log}(\cM, \tau)$, we have
\begin{align*}
\left\|J(x)\right\|_{\log} & \stackrel{\eqref{log-norm}}{=} \int\limits_0^\infty \left(1+\mu(t; J(x))\right)dt\stackrel{\eqref{pre-sin}}{=}\int\limits_0^\infty \left(1+\mu(t; x)\right)dt=\left\|x\right\|_{\log}.
\end{align*}
Thus, $J$ is an isometry from $\cL_{\log}(\cM, \tau)$ into $\cL_{\log}(\cN, \nu).$

Below we need the following property of Jordan $^*$-homomorphisms (see e.g.  \cite[P.212]{BR87}   for the case of bounded operators):
for any $x, y \in S(\cM, \tau)$ we have
\begin{align}\label{xyx}
\Phi(xyx)=\Phi(x)\Phi(y)\Phi(x).
\end{align}

Let $w\in \cN$ be a partial isometry in $\cN$ such that
\begin{align}\label{lJx}
w^*w & =J(\mathbf{1}).
\end{align}
For any $x\in \cL_{\log}(\cM, \tau)$, we have
\begin{align*}
|wJ(x)|^2 & ~=~J(x)^*w^*wJ(x)\stackrel{\eqref{lJx}}{=}J(x)^*J(\mathbf{1})J(x)=J(\mathbf{1}x\mathbf{1})^*J(\mathbf{1})J(\mathbf{1}x\mathbf{1})\\
& \stackrel{\eqref{xyx}}{=} J(\mathbf{1})J(x)^*J(\mathbf{1})^3J(x)J(\mathbf{1})=\left(J(\mathbf{1})J(x)J(\mathbf{1})\right)^*\left(J(\mathbf{1})J(x)J(\mathbf{1})\right)
\\ & \stackrel{\eqref{xyx}}{=} J(x)^*J(x)=|J(x)|^2,
\end{align*}
and hence,
\begin{align*}
\left\|wJ(x)\right\|_{\log} & \stackrel{\eqref{log-norm}}{=} \left\||wJ(x)|\right\|_{\log}=\left\||J(x)|\right\|_{\log}=\left\|J(x)\right\|_{\log}=\left\|x\right\|_{\log}.
\end{align*}
Hence, the mapping
\begin{align*}
x\in \cL_{\log}(\cN, \nu) \mapsto wJ(x)\in \cL_{\log}(\cN,\nu)
\end{align*}
defines an isometry from $\cL_{\log}(\cM, \tau)$ into $\cL_{\log}(\cN, \nu).$
That is, we have proved the following result.

\begin{prop}\label{jor-iso}
The mapping defined as in \eqref{gen-form} is an isometry from
$\cL_{\log}(\cM,\tau)$ into $\cL_{\log}(\cN, \nu)$.
\end{prop}

\section{Isometries of Haagerup-Schultz algebras}\label{IsoHSa}
Throughout this section, we always assume that 
  $\cM$ is a semifinite von Neumann algebra equipped with a semifinite faithful normal trace $\tau$. 

\subsection{Every $\log$-isometry preserves singular value functions of integrable  operators}\label{subsec-1}

Firstly, we obtain the following   description of the log-norm in terms of the distribution function.

\begin{lemma}\label{key-lm1-1} Let  $x \in \cL_{\log}(\cM, \tau)$. Then
\begin{align}\label{key-eq-1}
\left\|\lambda x\right\|_{\log} & =\int\limits_0^\infty \frac{\lambda}{1+\lambda t}d\left(t; |x|\right) \, dt
\end{align}
for all $\lambda > 0$.
\end{lemma}

\begin{proof} Let  $x \in \cL_{\log}(\cM, \tau)$.
We claim  that
\begin{align}\label{dphidpsi}
\left\|x\right\|_{\log} & = \int\limits_0^\infty \frac{d(s; |x|)}{1+s} \, ds.
\end{align}
Indeed, 
the functions $\phi(s)=\log(1+s),\, s\ge 0$ and $f(t)=\mu(t; x),\, t\ge 0$ satisfy all conditions of Proposition~\ref{phi-f}. Thus,
\begin{align*}
\left\|x\right\|_{\log} & \stackrel{\eqref{log-norm}}{=}
\int\limits_0^\infty \log\!\bigl(1+\mu(t;x)\bigr)\,dt\stackrel{\eqref{phi|f|}}{=}\int\limits_0^\infty \frac{1}{1+s}d(s;\mu(x))\, ds
\stackrel{\eqref{dpdmu}}{=}\int\limits_0^\infty \frac{d(s;|x|)}{1+s}\, ds.
\end{align*}
This proves the claim. 

Using~\eqref{dphidpsi} and~\eqref{dslam}, we have
\begin{eqnarray*}
\left\|\lambda x\right\|_{\log} & \stackrel{\eqref{dphidpsi}}{=}&\int\limits_0^\infty \frac{1}{1+s}d(s; \lambda |x|) \, ds
\stackrel{\eqref{dslam}}{=} \int\limits_0^\infty \frac{1}{1+s}d\left(\frac{s}{\lambda}; |x|\right) \, ds \\
& \stackrel{\{s=\lambda t\}}{=}&   \int\limits_0^\infty \frac{\lambda}{1+\lambda t}d\left(t; |x|\right) \, dt.
\end{eqnarray*}
for all $\lambda > 0$.
 The proof is complete.
\end{proof}

\begin{lemma}\label{key-lm1}
Let $\phi_1, \phi_2$ be positive measurable functions on $(0, \infty)$ satisfying
\begin{align}\label{key-ineq1}
\int\limits_0^\infty \frac{\phi_1(s)}{1+\lambda s}\,ds & = \int\limits_0^\infty \frac{\phi_2(s)}{1+\lambda s}\,ds
\end{align}
for all $\lambda > 0$. If $\phi_1$ is integrable on $(0,\infty)$, then $\phi_2$ is  also integrable on $(0, \infty)$. 
Moreover, $\phi_1 =\phi_2$ almost everywhere on $(0,\infty)$. 
\end{lemma}

\begin{proof} Take  $\lambda_n = \frac1n$ for $n\ge 1$.
Since $\lambda_n \downarrow_n 0$ and $\phi_i \geq 0$, $i=1,2$, it follows that for each $s>0$, we have 
\begin{align*}
\frac{1}{1+\lambda_n s}\,\phi_i(s) \uparrow_n \phi_i(s) \quad \text{for all }  \, i=1,2.
\end{align*}
By the Monotone Convergence Theorem \cite[Theorem 4.20]{Klenke}, we have 
\begin{align*}
\int\limits_0^\infty \phi_2(s)\,ds & ~=~\lim\limits_{n \to \infty}\int\limits_0^\infty  \frac{1}{1+\lambda_n s}\,\phi_2(s)\,ds  \\
& \stackrel{\eqref{key-ineq1}}{=} \lim\limits_{n \to \infty}\int\limits_0^\infty  \frac{1}{1+\lambda_n s}\,\phi_1(s)\,ds =\int\limits_0^\infty \phi_1(s)\,ds<\infty.
\end{align*}
Thus,  $\phi_2$ is integrable on $(0, \infty)$.

 Recall that the map~\cite[p.125]{Widder71}
\begin{align}\label{Stieltjes}
\phi\in L_1(0;\infty) \mapsto \left[\zeta \mapsto \int\limits_0^\infty \frac{\phi(s)}{\zeta+s} \, ds\right]
\end{align}
is the (absolutely continuous) Stieltjes transform of the integrable function $\phi$ defined on $(0,\infty)$.

Changing $\lambda:=\frac{1}{\zeta}$ in \eqref{key-ineq1}, we have
\begin{align}\label{St=}
\int\limits_0^\infty \frac{\phi_1(s)}{\zeta+s}\, ds  = \int\limits_0^\infty \frac{\phi_2(s)}{\zeta+s}\, ds,
\end{align}
that is, the Stieltjes transforms of these functions coincide~(see~\eqref{Stieltjes}).
Since both $\phi_1$ and $\phi_2$ are integrable on $(0, \infty)$, it follows that 
their Stieltjes transforms in \eqref{St=} are well defined.
By the injectivity of the Stieltjes transform~\cite[Theorem 1.6]{Widder} (see also \cite[Chapter 6, Theorem 4]{Widder71}), we have
$\phi_1 =\phi_2$ almost everywhere on $(0,\infty)$.
The proof is complete.
\end{proof}

The following result demonstrates that a $\log$-isometry preserves the singular value functions of integrable operators, which is a crucial step in proving the main theorem.

\begin{lemma}\label{singular-function}
Let  $x\in L_1(\cM, \tau)$ and $y \in \cL_{\log}(\cN, \nu)$ be such that
\begin{align}\label{main-eq}
\left\|\lambda x \right\|_{\log} & = \left\|\lambda y\right\|_{\log}
\end{align}
for all $\lambda > 0$. Then
$\mu(x)=\mu(y)$. Furthermore, the restriction $\Phi|_{L_1(\cM, \tau)}$ is an isometry from $\left(L_1(\cM, \tau), \left\|\cdot \right\|_1\right)$
into $\left(L_1(\cN, \nu), \left\|\cdot\right\|_1\right).$
\end{lemma}

\begin{proof} Since $x\in L_1(\cM, \tau)$, it follows that $d(\cdot;|x|)$ is integrable on $(0;\infty)$ (see \eqref{iden}).
We have
\begin{align*}
\int\limits_0^\infty \frac{d(s;|y|)}{1+\lambda s}\, ds &  \stackrel{\eqref{key-eq-1}}{=} \frac{\left\|\lambda y\right\|_{\log}}{\lambda}\stackrel{\eqref{main-eq}}{=}\frac{\left\|\lambda x\right\|_{\log}}{\lambda}=\int\limits_0^\infty \frac{d(s;|x|)}{1+\lambda s}\, ds
\end{align*}
for all $\lambda>0$. By Lemma~\ref{key-lm1}, $d(t; |x|)=d(t; |y|)$ for almost all $t\in (0,\infty)$, and hence,
$\mu(x)=\mu(y)$.

Finally, we have
\begin{align*}
\left\|x\right\|_1 & =\int\limits_0^\infty \mu(t;x)dt=\int\limits_0^\infty \mu(t;y)dt=\left\|y\right\|_1
\end{align*}
for all $x\in L_1(\cM, \tau)$.  Setting $y=\Phi(x)$, we see that $\Phi|_{L_1(\cM, \tau)}$ is an isometry from $\left(L_1(\cM, \tau), \left\|\cdot \right\|_1\right)$
into $\left(L_1(\cN, \nu), \left\|\cdot\right\|_1\right).$
The proof is complete.
\end{proof}

\begin{remark}
The space $S(\cM, \tau)$ can be equipped with the following symmetric $F$-norm \cite{BHS24} via singular value function (see~\eqref{mu(x)}):
\begin{align*}
\left\|x\right\|_{L_0} & =\inf\limits_{t>0}\left\{t+\mu(t;x)\right\},\, x \in S(\cM, \tau).
\end{align*}
It should be noted that   Lemma~\ref{singular-function} 
fails if we replace  $\left\|{\cdot}\right\|_{\log}$ with $\left\|{\cdot}\right\|_{L_0}$ 
(see Example~4.6~in~\cite{BHS24}).
\end{remark}

\subsection{Some remarks concerning $\cL_{\log}(\cM,\tau)$}

(a) Substituting $\lambda:=\frac{1}{\zeta}$ in \eqref{key-eq-1}, we have
\begin{align*}
\left\|\zeta^{-1} x\right\|_{\log} & =\int\limits_0^\infty \frac{d\left(t; |x|\right)}{\zeta+t} \, dt
\end{align*}
for all $\zeta > 0$. In other words,
\begin{align*}
\zeta \in (0,\infty) \mapsto \left\|\zeta^{-1} x\right\|_{\log}
\end{align*}
coincides with the Stieltjes transform of the distribution function of $|x|$ (see \eqref{Stieltjes}).

(b) 
Let $\mathcal{M}$ be von Neumann algebra with a faithful normal tracial state $\tau$. For an operator of the form $1+x$, where $x\in\mathcal{L}_{\log}(\mathcal{M},\tau)$, the Fuglede--Kadison determinant of $x$ is defined as follows\cite{Br, DDSZ20}:
\begin{equation*}
    \det (1+x)=\exp (\tau(\log\abs{1+x})).
\end{equation*}
The following formula provides a method for computing the Fuglede--Kadison determinant of invertible operators via distribution function of $|x|$.
For any invertible $x \in \cM$ we have  
\begin{align}\label{FKd}
\det(x)  & = \delta \exp\left(\int\limits_{0}^{\infty} \frac{d(t; |x|-\delta\mathbf{1})}{\delta +  s}\, ds\right),
\end{align}
where $\delta$ is a positive number with $|x| \ge \delta\mathbf{1}$.
Indeed, 
setting $\lambda=1$ and $z=|x|\ge 0$ in~\eqref{dphidpsi},
we obtain
\begin{align}\label{fkk-1}
\log\!\big( \det(\mathbf{1} + z) \big)
= \tau\!\big( \log(\mathbf{1} + z) \big)
= \left\|z\right\|_{\log}
\stackrel{\eqref{dphidpsi}}{=}
\int\limits_{0}^{\infty} \frac{d(s; z)}{1 +  s}\, ds.
\end{align}
Since $|x|\ge \delta \mathbf{1}$, we have
\begin{align*}
\log\!\big(\det(\delta^{-1}x) \big) & =\log\!\big(\det(\mathbf{1} + (\delta^{-1}|x|-\mathbf{1})) \big)
\stackrel{\eqref{fkk-1}}{=}
\int\limits_{0}^{\infty} \frac{d(s; \delta^{-1}|x|-\mathbf{1})}{1 +  s}\, ds\\
& \stackrel{\eqref{dslam}}{=}\int\limits_{0}^{\infty} \frac{d(\delta s; |x|-\delta \mathbf{1})}{1 +  s}\, ds=\int\limits_{0}^{\infty} \frac{d(t; |x|-\delta \mathbf{1})}{\delta +  t}\, dt,
\end{align*}
which proves~\eqref{FKd}.

Assume $|x|=\sum\limits_{i=1}^n \lambda_i p_i$, where $\lambda_1\ge  \ldots \ge \lambda_n>0$ and
$p_ip_j=0$ for $i\neq j$, $\sum\limits_{i=1}^np_i=\mathbf{1}$. Since
\begin{align*}
d(t; |x|-\delta) =\begin{cases}
			0, & \text{if $t\ge \lambda_1-\delta$}\\
            \sum\limits_{i=1}^j \tau(p_i), & \text{if $\lambda_{j+1}-\delta\le t<\lambda_j-\delta$\, $(j=1, \cdots, n-1)$}\\
            1, & \text{if $0\le t <\lambda_n-\delta$,}
		 \end{cases}
\end{align*}
where $0<\delta<\lambda_n$, it follows from  \eqref{FKd} that\footnote{
See Part (d) for a detailed proof for a similar result. }
\begin{align*}
\det(x) & =\lambda_1^{\tau(p_1)}\cdots \lambda_n^{\tau(p_n)}.
\end{align*}

(c) 
The following formula provides a method for computing the Fuglede--Kadison determinant of positive elements in  Haagerup--Schultz algebras. 
Let $x \in \mathcal{L}_{\log}(\mathcal{M}, \tau)$ be a positive element. Then, for all $\zeta > 0$, we have
\begin{align}\label{FFK-}
\det\!\left( \mathbf{1} + \zeta^{-1} x \right)
= \exp\!\left( \int\limits_{0}^{\infty} \frac{d(s; |x|)}{\zeta + s} \, ds \right),
\end{align}
Indeed, by~\eqref{key-eq-1} we obtain
\begin{align*}
\log\!\big( \det(\mathbf{1} + \lambda x) \big)
= \tau\!\big( \log(\mathbf{1} + \lambda x) \big)
= \left\|\lambda x\right\|_{\log}
\stackrel{\eqref{key-eq-1}}{=}
\int\limits_{0}^{\infty} \frac{\lambda}{1 + \lambda s} \, d(s; |x|) \, ds,
\end{align*}
substituting variables via $\lambda = \frac{1}{\zeta}$, we obtain  \eqref{FFK-}.

Let $H$ be a separable Hilbert space, and set $\mathcal{M} = B(H)$, the algebra of all bounded operators on $H$. Let $\tau = \mathrm{Tr}$ denote the canonical trace on $B(H)$.

  Consider a positive trace-class operator   $x \in B(H)$, and let $\{\lambda_n\}_{n \ge 1}$ denote its eigenvalues, counted with multiplicities and ordered non-increasingly, that is, $\lambda_n \downarrow 0$  and $\sum\limits_{n\geq 1}\lambda_n<\infty$.
Since
\begin{align*}
d(\cdot; x) = \sum_{n \ge 1} n \, \chi_{[\lambda_{n+1}, \lambda_{n})},
\end{align*}
formula~\eqref{FFK-} yields the following classical expression for the (Fredholm) determinant of $x$ (see \cite[Chapter~3]{Simon05}):
\begin{equation*}
    \det(\mathbf{1}+\zeta x)=\prod\limits_{n\geq1}(1+\zeta \lambda_n).
\end{equation*}
Indeed, for any $n\ge2$, we have
\begin{eqnarray*}
\int\limits_{\lambda_n}^{\lambda_1}\frac{d(s;\abs{x})}{\zeta+s}\,ds&=&\sum\limits_{i=1}^{n-1}\int\limits_{\lambda_{i+1}}^{\lambda_{i}}\frac{d(s;\abs{x})}{\zeta+s}\,ds=\sum\limits_{i=1}^{n-1}\int\limits_{\lambda_{i+1}}^{\lambda_{i}}\frac{i}{\zeta+s}\,ds\\
&\stackrel{\tiny \left[s=\zeta t\right]}{=}&\sum\limits_{i=1}^{n-1}\int\limits_{\zeta^{-1}\lambda_{i+1}}^{\zeta^{-1}\lambda_{i}}\frac{i}{1+t}\,dt\\
&=&\sum\limits_{i=1}^{n-1}i\log(1+\zeta^{-1}\lambda_i)-\sum\limits_{i=1}^{n-1}i\log(1+\zeta^{-1}\lambda_{i+1})\\
&=&\sum\limits_{i=1}^{n-1}\log(1+\zeta^{-1}\lambda_i)-(n-1)\log(1+\zeta^{-1}\lambda_n).
\end{eqnarray*}
Since $0\le (n-1)\log(1+\zeta^{-1}\lambda_n)\le (n-1)\zeta^{-1}\lambda_n\to 0$ as $n\to \infty$ (see \cite[p.256]{Bartle1992}), we have 
\begin{align*}
\int\limits_0^\infty\frac{d(s;\abs{x})}{\zeta+s}\,ds & =\lim\limits_{n\to\infty}\int\limits_{\lambda_n}^{\lambda_1}\frac{d(s;\abs{x})}{\zeta+s}\,ds =\sum\limits_{n=1}^\infty\log(1+\zeta^{-1}\lambda_i).
\end{align*}
Thus,  
\begin{align*}
 \det(1+\zeta^{-1}x) & ~\stackrel{\eqref{FFK-}}{=}\exp\left(\int\limits_{0}^{\infty}\frac{d(s;x)}{\zeta+s}\,ds\right)\\
 & \quad =~\exp\left(\sum\limits_{n=1}^{\infty}\log(1+\zeta^{-1}\lambda_n)\right)=\prod\limits_{n\geq 1}(1+\zeta^{-1}\lambda_n).  
\end{align*}
Changing $\zeta$ to $\zeta^{-1}$, we obtain the required relation.

(d) The space $\cL_{\log}(\cM, \tau)$ can be defined as follows
\begin{align*}
\cL_{\log}(\cM, \tau) & = \left\{x\in S(\cM, \tau):     \int\limits_0^\infty \frac{d(s; x)}{1+s}ds<+\infty\right\}.
\end{align*}

\subsection{$\log$-isometry preserves partial isometries with trace-finite supports}
\label{sin-val-pre}

An operator $u \in B(H)$ is a partial isometry if $u$ is isometric on $\ker(u)^\perp$, that is,
$\left\|u(\xi)\right\|_H=\left\|\xi\right\|_H$ for all $\xi \in \ker(u)^\perp.$

Recall the following characterizations of partial isometries (see \cite[Theorem 2.3.3]{Murphy}). For an operator $u \in B(H)$ the following conditions are equivalent:
\begin{itemize}
\item[(i)] $u=uu^*u$;
\item[(ii)] $u^*u$ is a projection;
\item[(iii)] $uu^*$ is a projection;
\item[(iv)] $u$ is a partial isometry;
\item[(v)] $u^*$ is a partial isometry.
\end{itemize}

Below we need the following sets
\begin{itemize}
\item[-] $P_\tau(\cM)$ the set of all projections $p\in P(\cM)$ with $\tau(p)<\infty$;
\item[-] $\cP\cU_\tau(\cM)$ the set of all partial isometries $u\in \cM$ with $\tau(uu^*)<\infty$;
\item[-] $\cF(\cM, \tau)$ the set of all elements $x\in \cM$ with $\tau(s(x))<\infty$.
\end{itemize}

\begin{lemma}\label{imageofproj} Let $(\cM, \tau)$, $(\cN, \nu)$ and $\Phi$  as in Theorem~\ref{log-isometry} and
let  $u\in \cP\cU_\tau(\cM).$ Then
the element
$\Phi(u)$ is  a partial isometry with $\nu(\Phi(u)^*\Phi(u))=\tau(u^*u)$.
\end{lemma}

\begin{proof} For any $u \in \cP\cU_\tau(\cM)$, we have $\left\|\lambda\Phi(u)\right\|_{\log}=\left\|\lambda u\right\|_{\log}$ for all $\lambda>0$. Using Lemma~\ref{singular-function}, we have
\begin{align*}
\mu(\Phi(u)) &  = \mu(u) \stackrel{{\rm Lemma}~\ref{mu-prop}(c)}{=}     \chi_{[0,\tau(uu^*))}
\end{align*}
Therefore, by Lemma~\ref{mu-prop}(c),  $|\Phi(u)|$ is a projection with 
$$
\nu(|\Phi(u)|)= \nu(|\Phi(u)|^2)=\nu(\Phi(u)^*\Phi(u))=\tau(u^*u),
$$
that is,
$\Phi(u)$ is a partial isometry  in $\cN.$
The proof is complete. 
\end{proof}

\subsection{Proof of Theorem~\ref{log-isometry}} \label{proo$F$-o$F$-m-th}

\begin{proof}[Proof of Theorem~\ref{log-isometry}]
The implication (2)$\Rightarrow$(1) is proved in Proposition~\ref{jor-iso}. Therefore, 
it suffices to  prove the implication
(1)$\Rightarrow$(2).

By Lemma~\ref{singular-function}, the restriction $\Phi|_{L_1(\cM, \tau)}$ is an isometry from $(L_1(\cM, \tau), \left\|\cdot\right\|_1)$
into $(L_1(\cN, \nu), \left\|\cdot\right\|_1).$
By \cite[Theorem 2]{Yeadon80}, there exist, uniquely, a partial isometry $w\in \cN$, an unbounded
positive self-adjoint operator $b$  affiliated with $\cN$, and a normal Jordan $^*$-isomorphism $J$ of $\cM$
onto a weakly closed $^*$-subalgebra of $\cN$ such that
\begin{align}\label{wwsb}
w^*w=J(\mathbf{1})=s(b),
\end{align}
\begin{align}\label{commute}
\text{every spectral projection of}\,\, b\,\, \text{commutes with}\, \, J(x)\,\,\text{for all}\,\,0\le x \in \cM,\qquad
\end{align}
\begin{align}\label{xnub}
\tau(x) & =\nu(bJ(x)) \quad \text{for all} \quad 0\le x\in \cM,
\end{align}
and
\begin{align}\label{xwb}
\Phi(x) & =wbJ(x) \quad \text{for all} \quad x\in L_1(\cM,\tau)\cap \cM.
\end{align}

We first show that $0 \le b \le \mathbf{1}$. For $n \ge 2$, set
\begin{align*}
q_n = e_{\left(1+\frac{1}{n},\,n\right)}(b), \qquad b_n = b q_n.
\end{align*}
Then
\[
\left(1+\frac{1}{n}\right) q_n \le b_n \le n q_n.
\]

For any \(p \in P_\tau(\mathcal{M})\), we have
\[
\Phi(p) q_n \stackrel{\eqref{xwb}}{=} w b J(p) q_n \stackrel{\eqref{commute}}{=} w  b q_n J(p)  = w  b_n J(p)   .
\]
Hence,
\begin{align*}
q_n \Phi(p)^* \Phi(p) q_n
&= J(p) b_n  w^* w   b_n J(p) \stackrel{\eqref{wwsb}}{=} J(p) b_n s(b) b_n J(p) \\
&= J(p) b_n^2 J(p) \stackrel{\eqref{commute}}{=} b_nJ(p)^2b_n=b_n J(p) b_n.
\end{align*}
By Lemma~\ref{imageofproj}, \(\Phi(p)^*\Phi(p)\) is a projection. Therefore, 
\[
q_n \Phi(p)^*\Phi(p) q_n \le q_n \mathbf{1}q_n=q_n,
\]
and thus
\[
b_n J(p) b_n \le q_n.
\]

Let \(\{p_i\}_{i \in I}\) be an increasing net of \(\tau\)-finite projections in \(\mathcal{M}\) with \(p_i \uparrow \mathbf{1}\). Since the map
\[
y \mapsto b_n y b_n
\]
is monotone continuous on \(\mathcal{N}_{sa}\) \cite[Proposition 3.3.7(ii)]{BCLSZ}, and \(J\) is normal,
\[
q_n \ge b_n J(p_i) b_n \uparrow_i b_n J(\mathbf{1}) b_n \stackrel{\eqref{wwsb}}{=}b_n s(b) b_n = b_n^2 \geq\left(1+\frac{1}{n}\right)^2 q_n.
\]
Thus,  \(e_{\left(1+\frac{1}{n},\,n\right)}(b)=q_n = 0\). Since \(n\) was arbitrary, we conclude \(0 \le b \le \mathbf{1}\).

We now prove that \(b\) is a projection. For each \(i\), we have 
\begin{align*}
\Phi(p_i)^*\Phi(p_i)
&\stackrel{\eqref{xwb}}{=} J(p_i) b w^* w b J(p_i)
\stackrel{\eqref{wwsb}}{=} J(p_i) b^2 J(p_i) 
\stackrel{\eqref{commute}}{=} b J(p_i)^2 b = b J(p_i) b.
\end{align*}
By Lemma~\ref{imageofproj}, each \(b J(p_i) b=\Phi(p_i)^*\Phi(p_i)\) is a projection. Since \(J\) is normal, it follows that 
\[
b J(p_i) b \uparrow b J(\mathbf{1}) b \stackrel{\eqref{commute}}{=} b s(b) b = b^2.
\]
Thus,
\(b^2\) is a projection as the supremum of projections. Since \(b \ge 0\), it follows that
$
b = b^2.
$
Hence \(b\) is a projection.

By \eqref{wwsb},
\begin{align}\label{wwsbb}
w^*w=J(\mathbf{1})=s(b)=b.
\end{align}
Thus
\begin{align*}
\Phi(x)=wbJ(x)=ww^*wJ(x)=wJ(x)
\end{align*}
for all $x\in L_1(\cM, \tau)\cap \cM$.
 By \cite[Proposition 3.4.16]{DPS}, $\cF(\cM, \tau)\subset L_1(\cM, \tau)\cap \cM$ is $\left\|\cdot\right\|_1$-dense in $L_1(\cM, \tau).$ 
  Hence, by \cite[Proposition 4.7]{DSZ16}, it is $\left\|\cdot\right\|_{\log}$-dense in $\mathcal{L}_{\log}(\cM, \tau).$
On the other hand, using \eqref{xnub}, we have
\begin{align*}
\tau(p) & \stackrel{~\eqref{xnub}}{=} \nu(bJ(p)) \stackrel{~\eqref{wwsbb}}{=} \nu(J(\mathbf{1})J(p))=\nu(J(\mathbf{1})J(p)J(\mathbf{1}))\stackrel{~\eqref{xyx}}{=}\nu(J(p))
\end{align*}
for all $p\in P_\tau(\cM)$.  
 By \cite[Proposition 3.4 and Theorem 3.5]{BHS24}, $J$ extends to a trace-preserving 
 Jordan $^*$-monomorphism from $S(\cM, \tau)$ into $S(\cN, \nu)$, which we still denote by $J$.
 Thus,
$\Phi(x)=w J(x)$ for all $x\in\mathcal{L}_{\log}(\cM, \tau)$. 
Hence,   $\Phi$ is in the form of \eqref{gen-form}.

Suppose that  $\Phi$ is surjective.  Since $\Phi^{-1}$ is also an isometry, we may apply Lemma~\ref{singular-function} to both $\Phi$ and $\Phi^{-1}$. It follows that the restriction $\Phi|_{L_1(\cM, \tau)}$ is an isometry from $\left(L_1(\cM, \tau), \left\|\cdot\right\|_1\right)$
onto $\left(L_1(\cN, \nu), \left\|\cdot\right\|_1\right).$
By \cite[Corollary of Theorem 2]{Yeadon80},  $w$ is a unitary and $J$ is a bijection from $\cM$ to $\cN$, and hence, it is a Jordan $^*$-isomorphism (which extends to a Jordan $^*$-isomorphism from $S(\cM,\tau)$ to $S(\cN,\nu)$, see \cite[Proposition 3.4(2)]{BHS24}).
The proof is complete.
\end{proof}

A bijection $\phi: P(\cM) \to P(\cN)$ is called an orthoisomorphism, if
\begin{align*}
\phi(p)\phi(q)=0 \quad \Leftrightarrow \quad pq=0.
\end{align*}

We establish the following criterion for the existence of an isometry between two Haagerup–Schultz algebras.

\begin{corollary}
Let $\mathcal{M}$ and $\mathcal{N}$ be von Neumann algebras  equipped with  faithful normal semifinite traces $\tau$ and $\nu$, respectively. Consider the following assertions:
\begin{enumerate}
\item The $F$-spaces $\left(\mathcal{L}_{\log}(\mathcal{M}, \tau), \left\|\cdot \right\|_{\log}\right)$ and $\left(\mathcal{L}_{\log}(\mathcal{N}, \nu), \left\|\cdot \right\|_{\log}\right)$ are isometric;
\item there exists a trace-preserving Jordan $^*$-isomorphism between $\cM$ and $\cN$;
\item there exists a  linear bijection from
$S(\mathcal{M}, \tau)$ onto  $S(\mathcal{N}, \nu)$ which preserves singular value functions;
\item there exists a trace-preserving orthoisomorphism between $P(\cM)$ and $P(\cN)$.
\end{enumerate}
Then the first three conditions are equivalent. Moreover, if $\mathcal{M}$ and $\mathcal{N}$ have no direct summands of type I$_2$, then all four conditions are equivalent. 

\end{corollary}
\begin{proof}
    
The equivalence $(1)\Leftrightarrow(2)$ follows from Theorem~~\ref{log-isometry}.
The implication $(2)\Rightarrow(3)$ follows from~\cite[Theorem 3.5]{BHS24} and
the implication $(3)\Rightarrow(1)$ follows from the definition of log-norm~\eqref{log-norm}.
Finally, the equivalence $(2)\Leftrightarrow(4)$ follows from~\cite[Corollary of Theorem 1]{Dye55}.
\end{proof}

In \cite{DDSZ20}, the authors extended the notion of the determinant function $\Delta$, originally introduced by Fack~\cite{Fack83} for $\tau$-compact operators, to a natural algebra of $\tau$-measurable operators affiliated with a semifinite von Neumann algebra. Their construction agrees with that of Haagerup and Schultz in the finite case, and the resulting determinant is shown to be submultiplicative. They introduced the space $L_{\log_+}(\mathcal{M}, \tau)$, which plays a central role in their analysis.

Let $\mathcal{M}$ be a von Neumann algebra equipped with a faithful normal semifinite trace $\tau$. Define
\[
L_{\log_+}(\mathcal{M}, \tau)
   = \left\{ x \in S(\mathcal{M}, \tau) : \log_+ |x| \in L_1(\mathcal{M},\tau) + \mathcal{M} \right\},
\]
where $\log_+ t = \max\{\log t,\,0\}$ for $t>0$. It is known \cite[Proposition~3.1]{DDSZ20} that
$L_{\log_+}(\mathcal{M}, \tau)$ is a $^*$-subalgebra of $S(\mathcal{M},\tau)$ and satisfies
\[
L_1(\mathcal{M},\tau) + \mathcal{M} \subseteq L_{\log_+}(\mathcal{M}, \tau) \subseteq S(\mathcal{M}, \tau).
\]
Also note that \cite[p. 815]{DDSZ20}
\[
L_{\log_+}(\mathcal{M},\tau)
    = \mathcal{L}_{\log}(\mathcal{M},\tau) + \mathcal{M},
\]
and that, for a tracial von Neumann algebra $\mathcal{M}$, we have \cite[p. 813]{DDSZ20}
\[
L_{\log_+}(\mathcal{M},\tau)
    = \mathcal{L}_{\log}(\mathcal{M},\tau).
\]

It is natural to refer to the algebra $L_{\log_+}(\mathcal{M},\tau)$ as the {\it unital Haagerup–Schultz} algebra associated with $(\cM,\tau)$.

For $x \in L_{\log_+}(\mathcal{M},\tau)$, set
\begin{equation*}
\left\|x\right\|_{\log_+}
   = \int\limits_0^1 \log\!\bigl(1 + \mu(s;x)\bigr)\, ds.
\end{equation*}
It is readily verified that $\norm{\cdot}_{\log_+}$ is an $F$-norm on $L_{\log_+}(\cM,\tau)$, see e.g. \cite[Corollary~5.4]{DS09}. 

Motivated by Theorem~\ref{log-isometry}, we formulate the following problem.

\begin{prob}
Let $\mathcal{M}$ be a von Neumann algebra equipped with a faithful normal semifinite trace $\tau$.
How to  describe the isometries of $L_{\log_+}(\mathcal{M}, \tau)$?
\end{prob}


\section{Connections with $L_p$-isometries} \label{S-l_p}

Let $\mathcal{M}$ be a von Neumann algebra equipped with a faithful  normal semifinite trace $\tau$. Let $\mathcal{E}$ be a linear subspace of $S(\mathcal{M},\tau)$ equipped with a complete (quasi-)norm $\left\|{\cdot}\right\|_{\mathcal{E}}$. We say that $\mathcal{E}$ is a symmetric (quasi-Banach) operator space (or noncommutative symmetric (quasi-Banach) space) if for $x\in\mathcal{E}$ and for every $y\in S(\mathcal{M},\tau)$ with $\mu(y)\leq\mu(x)$, we have $y\in\mathcal{E}$ and $\left\|{y}\right\|_{\mathcal{E}}\leq\left\|{x}\right\|_{\mathcal{E}}$.

Recall the following construction of a symmetric (quasi-)Banach operator space (or noncommutive symmetric quasi-Banach space) $E(\mathcal{M},\tau)$ from \cite{KS,S14} (see  also its detailed exposition in \cite{LSZ}). Let $E(0,\infty)$ be a symmetric (quasi-)Banach function space on $(0,\infty)$. Set 
\begin{equation*}
    E(\mathcal{M},\tau)=\left\{x\in S(\mathcal{M},\tau):\mu(x)\in E(0,\infty)\right\}
\end{equation*}
equipped with the natural 
(quasi-)norm defined by
\begin{equation*}
    \left\|{x}\right\|_{E(\mathcal{M},\tau)}\triangleq \left\|{\mu(x)}\right\|_{E(0,\infty)},\quad x\in E(\mathcal{M},\tau).
\end{equation*}

If $E(0,\infty)=L_p(\mathcal{M},\tau),\,0<p<\infty$, then the   space $E(\mathcal{M},\tau)$ is the so-called noncommutative $L_p$-space given by
\begin{equation*}
    L_p(\mathcal{M},\tau)\triangleq \left\{x\in S(\mathcal{M},\tau):\left\|{x}\right\|_p\triangleq\tau(\abs{x}^p)^{\frac{1}{p}}<\infty\right\}.
\end{equation*}
\begin{corollary}\label{cor4.10}
    Let $\Phi$ be a continuous linear operator from $\mathcal{L}_{\log}(\mathcal{M},\tau)$ to $\mathcal{L}_{\log}(\mathcal{N},\nu)$. Then, the following statements are equivalent:
    \begin{enumerate}[label=(\arabic*)]
        \item$\Phi$ is an isometry (not necessarily surjective) from $(\mathcal{L}_{\log}(\mathcal{M},\tau),\left\|{\cdot}\right\|_{\log})$ into $(\mathcal{L}_{\log}(\mathcal{N},\nu),\left\|{\cdot}\right\|_{\log})$;
        \item\begin{equation*}
            \Phi(x)=wJ(x),\quad\forall x\in\mathcal{L}_{\log}(\mathcal{M},\tau),
            \end{equation*}
        where $J:\mathcal{M}\to\mathcal{N}$ is a normal Jordan $^*$-monomorphism, the restriction $J|_{\mathcal{M}\cap\mathcal{L}_{\log}(\mathcal{M},\tau)}$ extends as a Jordan $^*$-monomorphism from $\mathcal{L}_{\log}(\mathcal{M}, \tau)$ into $\mathcal{L}_{\log}(\mathcal{N}, \nu)$ satisfying $\nu(J(x)) = \tau(x)$   for all $0\le x \in \cM$ and
    $w$ is a partial isometry in $\mathcal{N}$ such that $w^*w=J(\mathbf{1})$;
        \item$\Phi$ preserves singular value functions;
        \item the restriction of $\Phi$ on  $L_p(\mathcal{M},\tau)\cap\mathcal{L}_{\log}(\mathcal{M},\tau)$   for all $0<p\leq\infty$  is an isometry with respect to $\left\|{\cdot}\right\|_p$;
        \item the restriction of $\Phi$ on $ L_p(\mathcal{M},\tau)\cap\mathcal{L}_{\log}(\mathcal{M},\tau)$  for all $0<p<\infty $ is an isometry  with respect to $\left\|{\cdot}\right\|_p$.

    \end{enumerate}
\end{corollary}

\begin{proof}
   The equivalence (1)$\Longleftrightarrow $(2) follows from Theorem \ref{log-isometry}. The implication (2)$\Longrightarrow$(3) follows from (\ref{pre-sin}).
   The implications (3)$\Longrightarrow$(1), (3)$\Longrightarrow$(4) and (4)$\Longrightarrow$(5) are trivial.
It suffices to     prove the implication (5)$\Longrightarrow$(3).

    Since $\tau$ is semifinite, it follows that 
    the net of all $\tau$-finite  projections $\{e_i\}_{i\in I}$ satisfies that  $e_i\uparrow\mathbf{1}$.

Since $\Phi$ is an isometry   on $ L_1(\mathcal{M},\tau)\cap\mathcal{L}_{\log}(\mathcal{M},\tau)$    with respect to $\left\|{\cdot}\right\|_1$, we may extend $\Phi$ to $L_1(\cM,\tau)$ (still denoted by $\Phi$). 
By \cite[Theorem 2]{Yeadon80}, there exists a partial isometry $w$, a positive self-adjoint operator $b$ and a Jordan $^*$-monomorphism from $\cM$ into $\cN$ such that 
$$\Phi(x) = w b J(x),~ \forall x\in L_1(\cM,\tau)\cap \cM. $$
Moreover, 
for any $\tau$-finite projection $e_i \in \cM$, we have 
$$b_i := b J(e_i )\in L_1(\cN,\nu), ~ w_i := w  J(e_i ), ~w_i^* w_i =J(e_i )=s(b_i ),$$
 $b_i $ commutes with $J(x)$, $x\in e_i \cM e_i$  
and 
$$
\Phi(x) = w_i b_i J(x), ~\forall x\in e_i \cM e_i
$$
(see \cite[p.46--47]{Yeadon80}).
Since $e_i \cM e_i \subset L_p(\cM,\tau)$ for all $0<p<\infty $, it follows that 
the mapping $\Phi: x \mapsto w_i b_i J(x)$ is also an isometry on $e_i \cM e_i$ with respect to $\left\|{\cdot}\right\|_p$. 

    Firstly, we claim that $b_i$ is a projection. Consider the spectral projection $e_{(1,\infty)}(b_i)$ of $b_i$.
    Assume that  
     $J(e_i \mathcal{M} e_i )\cdot e_{(1,\infty)}(b_i)\ne0$.
     There exists a $\tau$-finite projection $e\in e_i\mathcal{M}e_i$ such that $f=J(e)$ and $f\cdot e_{(1,\infty)}(b_i)\ne0$.  
     Hence, 
     $$\nu(b_i^p\cdot f)=  \nu (b_i^p\cdot f^p)=\nu ((b_if)^p)= \left\|{b_i f}\right\|_p^p =\left\|{e}\right\|_p^p =\tau(e)<\infty$$ for any $0<p<\infty$.
Observe that $$\nu (b_i^p\cdot f)\ge \nu(b_i^p e_{(1,\infty)}(b_i)\cdot f)\to \infty$$ 
as   $p\to\infty$, which is  a contradiction with $  \nu (b_i^p\cdot f)= \tau(e ) <\infty$.
Hence, $ J(e_i \cM e_i )e_{(1,\infty)}(b_i)=0$. 
Therefore, 
    \begin{equation*}
        \tau(e)=\nu(b_i^pe_{(0,1)}(b_i)\cdot f)+\nu (b_i^pe_{\{1\}}(b_i)\cdot f)
    \end{equation*}
    Since $ \mu( f\cdot  b_i^pe_{(0,1)}(b_i) \cdot  f ) \le  \mu (f\cdot  b_ie_{(0,1)}(b_i) \cdot f )\in L_1(0,\infty ) $ when $p\ge 1$ and $b_i^pe_{(0,1)}(b_i) \to 0$ as $p\to \infty $ in the measure topology\cite[Corollary 2.8.2]{DPS}, it follows that 
     $$\nu (b_i^pe_{(0,1)}(b_i)\cdot f)
     =\nu(f \cdot b_i^pe_{(0,1)}(b_i)\cdot f) 
    \stackrel{\mbox{\tiny \cite[Theorem 3.4.21]{DPS}}}{\to} 0$$ as  $p\to\infty$. Hence, 
    $\tau(e)=\nu (e_{\{1\}}(b_i)f)$ and $b_i=e_{\{1\}}(b_i)$ is a projection. In particular, $b_i=r(w_i)=w_i^* w_i =J(e_i  )$.  
Moreover, 
for any $x\in e_i \cM e_i$, we have 
$$J(x) \stackrel{\eqref{xyx}}{=}J(e_i) J(x) J(e_i),$$
i.e., 
\begin{align}\label{Jeil}
J(e_i) \ge l(J(x)), r(J(x)).
\end{align}
Hence, we have 
$$
\Phi(x) = w_i b_i J(x) = w_i s(b_i) J(x) = w_iw_i^* w_i J(x) = w_i J(x)
$$
for all $x\in e_i \cM e_i$. Since 
\begin{align*}
\mu(J(x)) & \geq\mu(w_iJ(x))\geq\mu(w_i^*w_iJ(x))=\mu(J(e_i)J(x))\stackrel{\eqref{Jeil}}{=}\mu(J(x)),
\end{align*}
it follows that $\mu(J(x))=\mu(\Phi(x))$.

Note that for any $\tau$-finite projection $e\in \cM$, we have 
$$
\left( \int\limits_0^\infty \mu(t;e)dt \right)^{1/p }= \left\|{e}\right\|_{p} =\left\|{\Phi(e)}\right\|_{p} = \left\|{J(e)}\right\|_p = \left(\int\limits_0^\infty\mu(t;J(e))dt\right)^{1/p}.
$$ 
By Lemma \ref{mu-prop}(c), we have 
$$
\int\limits_0^{\infty}\chi_{[0,\tau(e))}\,dt=\int\limits_0^{\infty}\chi_{[0,\nu(J(e)))}\,dt.
$$
Hence, for any $\tau$-finite projection $e\in\cM$, we have $\tau(e)=\nu(J(e))$. By (\ref{pre-sin}), $J$ preserves singular value functions, that is,
for each $e_i$ and any $x\in e_i \cM e_i$, we have 
$$ \mu(x) =\mu(J(x))=\mu(\Phi(x)).  $$
In other words, for any $x\in \cF(\cM,\tau)$, we have 
$$\mu(x) =\mu(\Phi(x)). $$

Recall that 
   $\cF(\cM, \tau)\subset L_1(\cM, \tau)\cap \cM$ is $\left\|\cdot\right\|_1$-dense in $L_1(\cM, \tau)$\cite[Proposition 3.4.16]{DPS} and  $\left\|\cdot\right \|_{\log}$-dense in $\mathcal{L}_{\log}(\cM, \tau)$\cite[Proposition 4.7]{DSZ16}. 
For any $x\in \cL_{\log }(\cM,\tau)$, there exists a sequence $\{x_n\}_{n\ge 1}\subset \cF(\cM,\tau)$ such that $x_n \to x$ in $\left\|{\cdot}\right\|_{\log}$ as $n\to \infty $. 
Moreover, since $\Phi$ is a continuous mapping on $\cL_{\log}(\cM,\tau)$, it follows that $\Phi(x_n)\to \Phi(x)$ as $n\to \infty $ in $\left\|{\cdot}\right\|_{\log}$. 
Since the topology induced by $\left\|{\cdot}\right\|_{\log}$ is stronger than the measure topology (see e.g. \cite[Remark 4.8]{DSZ16} or \cite[Lemma 2.4]{HLS}), it follows that 
$\{x_n\}_{n\ge 1}$ such that $x_n \to x$ and  $\Phi(x_n)\to \Phi(x)$  as $n\to \infty $ in the measure topology. 
Since $\mu(x_n)=\mu(\Phi(x_n))$ for all $n\ge 1$, it follows from \cite[Proposition 3.2.11]{DPS}
that 
$$\mu(x) =\mu(\Phi(x)). $$
This completes the proof. 
\end{proof}

\begin{corollary}\label{cor:surjective}
    Let $\Phi$ be a continuous linear operator from $\mathcal{L}_{\log}(\mathcal{M},\tau)$ to $\mathcal{L}_{\log}(\mathcal{N},\nu)$.    Then, the following statements are equivalent:
    \begin{enumerate}[label=(\arabic*)]
        \item $\Phi$ is a surjective isometry from $(\mathcal{L}_{\log}(\mathcal{M},\tau),\left\|{\cdot}\right\|_{\log})$ onto $(\mathcal{L}_{\log}(\mathcal{N},\nu),\left\|{\cdot}\right\|_{\log})$;
        \item \begin{equation*}
            \Phi(x)=uJ(x),~ \forall x\in\mathcal{L}_{\log}(\mathcal{M},\tau),
        \end{equation*}
        where $u$ is a unitary in $\cN$ and $J:\cM\to\cN$ is a normal Jordan $^*$-isomorphism, the restriction $J|_{\mathcal{M}\cap\cL_{\log}(\cM,\tau)}$ extends as a Jordan $^*$-isomorphism from $\cL_{\log}(\cM,\tau)$ into $\cL_{\log}(\cN,\nu)$ satisfying $\nu(J(x))=\tau(x)$ for all $0\leq x\in\mathcal{M}$
        \item $\Phi$ is surjective and preserves the singular value functions;
          \item the restriction of $\Phi$ on  $L_p(\mathcal{M},\tau)\cap\mathcal{L}_{\log}(\mathcal{M},\tau)$  
          is a surjective isometry from $L_p(\mathcal{M},\tau)\cap\mathcal{L}_{\log}(\mathcal{M},\tau)$ onto $$L_p(\mathcal{N},\nu)\cap\mathcal{L}_{\log}(\mathcal{N},\nu)$$
          for all $0<p\leq\infty$  with respect to $\left\|{\cdot}\right\|_p$;
        \item the restriction of $\Phi$ on  $L_p(\mathcal{M},\tau)\cap\mathcal{L}_{\log}(\mathcal{M},\tau)$  
          is a surjective isometry from $L_p(\mathcal{M},\tau)\cap\mathcal{L}_{\log}(\mathcal{M},\tau)$ onto $$L_p(\mathcal{N},\nu)\cap\mathcal{L}_{\log}(\mathcal{N},\nu)$$
          for all $0<p< \infty$  with respect to $\left\|{\cdot}\right\|_p$;
        \item there exist $0<p\ne q<\infty$    such that the restriction of $\Phi$ on $L_p(\mathcal{M},\tau)$ is a surjective isometry from $((L_p(\mathcal{M},\tau)\cap\mathcal{L}_{\log}(\mathcal{M},\tau),\left\|{\cdot}\right\|_p)$ onto $(L_p(\mathcal{N},\nu)\cap\mathcal{L}_{\log}(\mathcal{N},\nu),\left\|{\cdot}\right\|_p)$ and the restriction of $\Phi$ on $L_q(\mathcal{M},\tau)$ is also a surjective isometry from $((L_q(\mathcal{M},\tau)\cap\mathcal{L}_{\log}(\mathcal{M},\tau),\left\|{\cdot}\right\|_q)$ onto $(L_q(\mathcal{N},\nu)\cap\mathcal{L}_{\log}(\mathcal{N},\nu),\left\|{\cdot}\right\|_q)$.
    \end{enumerate}
\end{corollary}

\begin{proof}  
The equivalence (1)$\Longleftrightarrow$(2) follows from Theorem~\ref{log-isometry}. 
The equivalence 
(1)$\Longrightarrow$(3) follows from Lemma~\ref{singular-function}. 
The implication
(3)$\Longrightarrow$(1) follows  from (\ref{log-norm}). 
The implications 
(3)$\Longrightarrow$(4)$\Longrightarrow$(5)$\Longrightarrow$(6) are trivial.
 

It suffices to prove the implication (6)$\Longrightarrow$(3).
Without loss of generality, we may assume that $p\ne 2$. 
Since $L_p(\mathcal{M},\tau)\cap\mathcal{L}_{\log}(\mathcal{M},\tau)$ is dense in $L_p(\cM,\tau)$, it follows that 
$\Phi$ extends to an isometry from $L_p(\cM,\tau)$ into $L_p(\cN,\nu)$ (still denoted by $\Phi$). 
We claim that $\Phi$ is a surjective isometry from $L_p(\cM,\tau)$ onto $L_p(\cN,\nu)$. 
Indeed, 
 for any $y\in L_p(\cN,\nu)$, there exists a sequence $\{y_n\}\in L_p(\mathcal{N},\nu)\cap\mathcal{L}_{\log}(\mathcal{N},\nu)$ such that $\left\|{y_n - y}\right\|_p\to 0$ as $n\to \infty$. 
Therefore, we have 
$$
\left\|{ \Phi^{-1}(y_n) - \Phi^{-1}(y)}\right\|_p\to 0
$$ as $n\to \infty$. 
Let $x$ be the limit of $\Phi^{-1}(y_n)$ in $L_p(\cM,\tau)$. 
We have 
$$
\Phi(x) = \lim_{n\to \infty} \Phi(\Phi^{-1}(y_n))  = \lim_{n\to \infty}  y_n =y.
$$
This shows that $\Phi$ is surjective. 

By \cite[Theorem 2]{Yeadon80}, there exists a  unitary operator  $w$, a positive self-adjoint operator $b$ and a Jordan $^*$-isomorphism from $\cM$ onto $\cN$ such that 
$$
\Phi(x) = w b J(x),~ \forall x\in L_p(\cM,\tau)\cap \cM.
$$
Moreover, 
for any $\tau$-finite projection $e_i \in \cM$, we have 
$$
b_i := b J(e_i ), ~ w_i := w  J(e_i ), ~w_i^* w_i =J(e_i )=s(b_i ),
$$
$b_i $ commutes with $J(x)$, $x\in e_i \cM e_i$,  
and 
$$
\Phi(x) = w_i b_i J(x), ~\forall x\in L_p(\cM,\tau)\cap e_i \cM e_i.
$$

We  claim  $b_i$ is a projection. 
Assume by contradiction that   $e_{(1,\infty)}(b_i)$ (or  $e_{(0,1)}(b_i) $)  is not $0$.
Since $b$ is affiliated with the center of $\cN$ \cite[Corollary]{Yeadon80}, it follows that $b_i =b J(e_i)$ is a $\tau$-measurable operator (in $L_p(\cN,\nu)$) affiliated with the center  of $J(e _i \cM e_i )$. 
Since $J$ is a Jordan $^*$-isomorphism from $e_i \cM e_i $ onto $J(e_i \cM e_i )$, it follows that  there exists a  projection $r\in e_i\mathcal{M}e_i$ such that $J(r)=e_{(1,\infty)}(b_i)\in J(e_i \mathcal{M}e_i)$.  

Since the restriction of $\Phi$ on $L_p(\mathcal{M},\tau)$ is an isometry, it follows that 
$\left\|{r}\right\|_p=\left\|{b_iJ(r)}\right\|_p$. 
On the other hand, 
$\Phi$ is an isometry on  $L_q(\mathcal{M},\tau)\cap \cM $ with respect to  $\left\|{\cdot}\right\|_q$. 
Therefore, we 
have $\left\|{r}\right\|_q=\left\|{b_iJ(r)}\right\|_q$. 
That is,
\begin{equation*}
    \tau(r)=\left\|{b_iJ(r)}\right\|_p^p=\left\|{b_iJ(r)}\right\|_q^q.
\end{equation*}
Then, $\nu(\abs{b_iJ(r)}^p)=\nu(\abs{b_iJ(r)}^q)$. 
 Therefore,   
\begin{align*}
    0 & =\nu\left(\abs{b_iJ(r)}^{q}\right)-\nu\left(\abs{b_iJ(r)}^{p}\right)=\nu\left((\abs{b_iJ(r)}^{q-p}-1)\abs{b_iJ(r)}^{p}\right)\\
    &=\nu\left(\left(\abs{b_i\cdot e_{(1,\infty)}(b_i)}^{q-p }-1\right)\abs{b_i\cdot e_{(1,\infty)}(b_i)}^{p}\right) \ne 0,
\end{align*}
which is a contradiction. 
Arguing similarly for $e_{(0,1)} (b_i)$, we obtain that $e_{(0,1)} (b_i)=0$.
This shows that 
$b_i$ is a partial isometry. 
Arguing similarly as the proof of Corollary \ref{cor4.10}, we obtain that  $\Phi(x)=wbJ(x)$ preserves the singular value functions of elements from $\mathcal{L}_{\log}(\mathcal{M},\tau)$, which completes the proof. 
\end{proof}

\begin{remark}
   Note that when $\Phi$ is injective, 
   the Condition (5) in Corollary \ref{cor4.10} can not be relaxed to ``the restriction of $\Phi$ on $ L_p(\mathcal{M},\tau)\cap\mathcal{L}_{\log}(\mathcal{M},\tau)$  is an isometry  with respect to $\left\|{\cdot}\right\|_p$  for all two different values of $p \in (0,\infty) $'' as Condition (6) in Corollary \ref{cor:surjective}. 
Indeed, 
    let $\mathcal{M}=\mathbb{C}$ and  $\mathcal{N}=L_\infty (0,\infty )$.  
    Define $\Phi(c)=cx$ for any $c\in\mathbb{C}$, where $$x=\frac{ \chi_{(0,1)}+2\cdot \chi_{(1,3)}}{3} . $$
    Note that $\left\|{x}\right\|_1=\left\|{x}\right\|_2=1$. We obtain that 
      $\Phi$ is both an $L_1$-isometry and an $L_2$-isometry, but $\mu(c)\ne\mu(\Phi(c))$.
\end{remark}

\begin{corollary}
Let $\Phi$ be a continuous linear operator from $\mathcal{L}_{\log}(\mathcal{M},\tau)$ to $\mathcal{L}_{\log}(\mathcal{N},\nu)$. Then, any of the equivalent statements in Corollary \ref{cor:surjective} is equivalent to the following:
\begin{enumerate}
\item[(7)] for all symmetric quasi-Banach function spaces $E(0,\infty)$, the restriction of $\Phi$ on $E(\mathcal{M},\tau)$ is necessarily a surjective isometry from $(E(\mathcal{M},\tau)\cap \mathcal{L}_{\log}(\mathcal{M},\tau),\left\|{\cdot}\right\|_E)$ onto $(E(\cN,\nu)\cap \mathcal{L}_{\log}(\mathcal{M},\tau),\norm{\cdot}_E)$ for $E(\mathcal{M},\tau)$ and $E(\mathcal{N},\nu)$ generated by $E(0,\infty)$.  
\end{enumerate}
\end{corollary}
\begin{proof}
    Since $L_p(0,\infty)$, $0<p\leq\infty$, is a symmetric quasi-Banach function space, it follows that the above statement implies (4) of Corollary \ref{cor:surjective}.

    It suffices to prove that condition (3) in Corollary \ref{cor:surjective} implies condition (7). Since $\Phi$ preserves singular value functions, it follows that the restriction of $\Phi$ on $E(\mathcal{M},\tau)$ is an isometry from $(E(\mathcal{M},\tau) \cap \cL_{\log }(\cM,\tau),\left\|{\cdot}\right\|_E)$ into $(E(\mathcal{N},\nu)\cap \cL_{\log }(\cN,\nu),\left\|{\cdot}\right\|_E)$ for any $E(\mathcal{M},\tau)$ generated by an arbitrary symmetric quasi-Banach function space $E(0,\infty)$.

It remains to    prove that $\Phi$ is surjective. Since $\Phi:\mathcal{L}_{\log}(\mathcal{M},\tau)\to\mathcal{L}_{\log}(\mathcal{N},\nu)$ is surjective and preserves the singular value functions, it follows that for any $y\in E(\mathcal{N},\nu)\cap\mathcal{L}_{\log}(\mathcal{N},\nu)$, the element $x\in   \mathcal{L}_{\log}(\mathcal{M},\tau)$ such that  $\Phi(x)=y$ satisfies $\mu(x)=\mu(y)\in E(0,\infty)$. 
Hence, $\Phi(E(\mathcal{M},\tau)\cap \cL_{\log }(\cM,\tau))=E(\mathcal{N},\nu)\cap \cL_{\log }(\cN,\nu)$, that is, $\Phi|_{E(\mathcal{M},\tau)}:E(\mathcal{M},\tau)\cap \cL_{\log }(\cM,\tau)\to E(\mathcal{N},\nu)\cap \cL_{\log }(\cN,\nu)$ is surjective, which completes the proof.
\end{proof}

\section{Comparisons of Haagerup-Schultz algebras and noncommutative $L_1$-spaces} \label{S-coincide}

In this section, we obtain a complete characterization of semifinite von Neumann algebras $(\cM,\tau)$ for which
the Haagerup--Schultz  algebra $\cL_{\log}(\cM,\tau)$ coincides with
the  space $L_1(\cM,\tau)$.
We show that equality holds if and only if $\cM$ is atomic and
there exists a positive constant $\delta>0$ such that every nonzero projection $p\in\cM$
satisfies $\tau(p)\ge \delta$.
In all other cases, including diffuse algebras and  algebras with atoms of arbitrarily small trace,
the inclusion $L_1(\cM,\tau)\subsetneq \cL_{\log}(\cM,\tau)$ is strict.

It is immediate that $\log(1+t)\le t$ for all $t\ge0$, whence
\begin{align}\label{l1llog}
L_1(\cM,\tau)\subseteq \cL_{\log}(\cM,\tau)
\quad\text{and}\quad
\left\|{x}\right\|_{\log}\le \left\|x\right\|_1,\qquad x\in L_1(\cM,\tau).
\end{align}
The main aim of this section  is to determine when the equality holds for the   inclusion \eqref{l1llog}.

\begin{theorem}\label{thm:main1}
Let $\cM$ be a semifinite von Neumann algebra with faithful normal semifinite trace $\tau$.
Then, 
\[
\cL_{\log}(\cM,\tau)=L_1(\cM,\tau)
\]
if and only if $\cM$ is atomic and there exists a constant $\delta>0$ such that
\begin{align}\label{eq:delta}
\tau (p)\ge \delta \qquad \text{for every nonzero projection } p\in\cM.
\end{align}
Moreover, in this case, the topologies induced by $\left\|\cdot\right\|_{\log}$ and $\left\|\cdot\right\|_1$ coincide.
\end{theorem}

\begin{lemma}\label{lm1} Let $\mathcal{M}$ be a von Neumann algebra with a  faithful normal semifinite trace $\tau$. 
Then,
\begin{align*}
\mathcal{L}_{\log}(\mathcal{M},\tau)\cap\mathcal{M} = L_1(\mathcal{M},\tau)\cap\mathcal{M}.
\end{align*}
\end{lemma}

\begin{proof} Taking into account \eqref{l1llog}, it suffices to show that
\begin{align*}
 \mathcal{L}_{\log}(\mathcal{M},\tau)\cap\mathcal{M} \subseteq L_1(\mathcal{M},\tau)\cap\mathcal{M}.
\end{align*}

Let $x\in \mathcal{M}$ be such that  $\tau(\log(1+|x|))<\infty$. For $t\ge0$, note that
\[
\log(1+t) \ge \frac{t}{1+t}.
\]
Indeed, define a function $f$ on $[0,\infty)$ by setting  $f(t) = \log(1+t) - \frac{t}{1+t}$. 
Then, $f(0)=0$ and $f'(t) = \frac{t}{(1+t)^2}\ge0$. In particular,  $f(t)\ge0$.

By functional calculus,
\begin{align*}
\log(1+|x|) \ge |x|(\mathbf{1}+|x|)^{-1}.
\end{align*}
For all $t\in[0,\left\|x\right\|_\cM]$, we have
\begin{align*}
\frac{t}{1+t} \ge \frac{t}{1+\left\|x\right\|_\cM}.
\end{align*}
Hence, 
\begin{align*}
|x|(\mathbf{1}+|x|)^{-1} \ge \frac{1}{1+\left\|x\right\|_\cM}\,|x|.
\end{align*}
By the positivity and normality of 
  $\tau$, we have 
\[
\tau(\log(1+|x|)) \ge \tau\!\left(|x|(\mathbf{1}+|x|)^{-1}\right)
\ge \frac{1}{1+\left\|x\right\|_\cM}\,\tau(|x|).
\]
Since $\tau(\log(1+|x|))<\infty$, it follows that $\tau(|x|)<\infty$. Therefore $x\in L_1(\mathcal{M},\tau)\cap \cM$.
The proof is complete.
\end{proof}

\begin{lemma}\label{lem:diffusefail}
If $\cM$ contains  a sequence
of mutually orthogonal projections $\{p_n\}_{n\ge 1}$ with $\tau(p_n)\downarrow 0$ (e.g. $\cM$ is atomless),
then $\cL_{\log}(\cM,\tau)\ne L_1(\cM,\tau)$.
\end{lemma}

\begin{proof}
Let $\{p_n\}$ be mutually orthogonal projections with $\tau(p_n)=\varepsilon_n\to0$.
If necessary, passing to subsequence we can assume that $\varepsilon_n\le 4^{-n}$ for all $n\ge1$.
We have\footnote{Since $\log\big(1+t\big) \le \sqrt{t}$ for all $t>0$, we have 
$\sum\limits_{n=1}^\infty \varepsilon_n \log(1+\varepsilon_n^{-1})
\le \sum\limits_{n=1}^\infty \varepsilon_n\frac{1}{\sqrt{\varepsilon_n}} 
\le\sum\limits_{n=1}^\infty \frac{1}{2^n}<\infty$.}
\begin{align*}
\sum_{n=1}^\infty \varepsilon_n \log(1+\varepsilon_n^{-1})<\infty.
\end{align*}
In particular, the operator
\begin{align*}
x=\sum_{n=1}^\infty \varepsilon_n^{-1}p_n
\end{align*}
is $\tau$-measurable with
\begin{align*}
\tau(\log(1+\abs{x}))=\sum\limits_{n=1}^\infty \varepsilon_n\log(1+\varepsilon_n^{-1})<\infty,
\quad\text{but}\quad
\tau(\abs{x})=\sum\limits_{n=1}^\infty \varepsilon_n \varepsilon_n^{-1}=\infty.
\end{align*}
Hence,  $x\in\cL_{\log}(\cM,\tau)\setminus L_1(\cM,\tau)$,
which proves $\cL_{\log}(\cM,\tau)\ne L_1(\cM,\tau)$.
\end{proof}

\begin{lemma}\label{cor:necessity}
If $\cL_{\log}(\cM,\tau)=L_1(\cM,\tau)$, then $\cM$ must be atomic
and there exists $\delta>0$ such that \eqref{eq:delta} holds.
\end{lemma}

\begin{proof} Assume $\cL_{\log}(\cM,\tau)=L_1(\cM,\tau)$. Applying Lemma~\ref{lem:diffusefail}, we obtain that $\cM$ is atomic, that is,
\[
\cM\simeq \bigoplus_{i\in I} B(H_i)
\]
with $\tau=\sum\limits_i \tau_i$,
where $\tau_i$ is a faithful normal semifinite trace on $B(H_i)$, $i\in I$. 
If there exists a sequence of atoms $p_i\in B(H_i)$ with $\tau(p_i)\downarrow0$,
then again by
Lemma~\ref{lem:diffusefail}, we have $\cL_{\log}(\cM, \tau)\neq L_1(\cM, \tau)$.
The proof is complete.
\end{proof}

\begin{lemma}\label{lem:suff}
Assume that $\cM$ is atomic and that there exists $\delta>0$ such that
$\tau(p)\ge\delta$ for every nonzero projection $p\in\cM$.
Then $\cL_{\log}(\cM,\tau) \subseteq L_1(\cM,\tau)$.
\end{lemma}

\begin{proof} Let \(x \in S(\mathcal{M}, \tau)\). By \cite[Proposition 2.4.2]{BCLSZ}, we have 
\[
\tau\!\left(e_{(\lambda, \infty)}(|x|)\right) \longrightarrow 0 \quad \text{as } \lambda \to \infty .
\]
Hence, there exists \(\lambda > 0\) such that
\[
\tau\!\left(e_{(\lambda, \infty)}(|x|)\right) \le \frac{\delta}{2}.
\]
By assumption, every non-zero projection \(p \in P(\mathcal{M})\) satisfies \(\tau(p) \ge \delta\). Therefore,
\(e_{(\lambda, \infty)}(|x|)=0\). Consequently,
\[
|x| \le \lambda \mathbf{1},
\]
which implies that \(x \in \mathcal{M}\). That is, 
\(S(\mathcal{M}, \tau) = \mathcal{M}\).
It follows that
\[
\mathcal{L}_{\log}(\mathcal{M}, \tau) \subseteq S(\mathcal{M}, \tau) = \mathcal{M}.
\]
Thus, by Lemma~\ref{lm1},  we have \begin{align*}
\mathcal{L}_{\log}(\mathcal{M},\tau)=\mathcal{L}_{\log}(\mathcal{M},\tau)\cap\mathcal{M} = L_1(\mathcal{M},\tau)\cap\mathcal{M}\subseteq L_1(\mathcal{M},\tau).
\end{align*}
The proof is complete.
\end{proof}

\begin{proof}[Proof of Theorem~\ref{thm:main1}]
Combining \eqref{l1llog}, Lemmas~\ref{cor:necessity} and~\ref{lem:suff} yields the equivalence:
\begin{align*}
\cL_{\log}(\cM,\tau)=L_1(\cM,\tau)
\quad\Longleftrightarrow\quad
\text{$\cM$ is atomic with \eqref{eq:delta} holding.}
\end{align*}

The inequality \eqref{l1llog} implies that the identity mapping on $L_1(\cM, \tau)$ is $\left\|\cdot\right\|_1$-$\left\|\cdot\right\|_{\log}$-continuous.
By the  Open Mapping Theorem for $F$-spaces (see e.g. \cite{KPR}),  it is also is $\left\|\cdot\right\|_{\log}$-$\left\|\cdot\right\|_1$-continuous,
which implies the topologies induced by $\left\|\cdot\right\|_{\log}$ and $\left\|\cdot\right\|_1$ coincide.
The proof is complete.
\end{proof}

\subsection{Isometries of Haagerup--Schultz algebras associated with a type I-factors}
Assume that $\cM = B(H)$ for a Hilbert space $H$, equipped with the standard trace ${\rm tr}$.
The space $\cL_{\log}(B(H), {\rm tr})$ forms a subalgebra of $K(H)$, the ideal of compact operators in $B(H)$.

For any operator $x \in \cL_{\log}(B(H), {\rm tr})$, we have (see \cite[Example~4.1.6]{BCLSZ})
\begin{align*}
\mu(x) & = \sum\limits_{n \ge 1} \mu_n \, \chi_{[n-1,\,n)} ,
\end{align*}
where $\mu_1 \ge \mu_2 \ge \cdots \ge 0$ is the sequence of nonzero singular values of $x$, each repeated according to its multiplicity.
By the definition of the logarithmic norm~\eqref{log-norm}, we  have 
\begin{align*}
\left\|x\right\|_{\log} & = \sum\limits_{n \ge 1} \log(1 + \mu_n).
\end{align*}

The Haagerup--Schultz algebra associated with $B(H)$ can be defined as follows
\begin{align*}
\mathcal{L}_{\log}(H) = \Bigl\{ x\in K(H) : \left\|x\right\|_{\log} = \sum_{n\ge1} \log(1+\mu_n) < \infty \Bigr\}.
\end{align*}

The following result  follows immediately from Theorem~\ref{thm:main1}.
\begin{prop}\label{coincide} Let $H$ be a Hilbert space. Then
\begin{enumerate}
\item the Haagerup--Schultz algebra associated with $B(H)$ coincides with the ideal of trace-class operators, i.e., $\mathcal{L}_{\log}(H)=L_1(H)$;

\item the topologies generated by $\left\|\cdot\right\|_{\log}$ and $\left\|\cdot \right\|_1$ coincide.
\end{enumerate}
\end{prop}

The following result is an immediate consequence of Theorem~\ref{log-isometry} and the general characterization of (anti-)$^*$-isomorphisms on $B(H)$~\cite[Theorem~A.8]{Molnar}.

\begin{corollary}\label{cor:log-iso}
Let $T$ be a surjective linear isometry from $\mathcal{L}_{\log}(H)$ onto itself with respect to the norm $\left\|\cdot \right\|_{\log}$.
Then there exist unitary operators $u, v \in B(H)$ such that
\begin{enumerate}
  \item either
  \begin{align*}
   T(x) = u x v
  \end{align*}
  for all $x \in\mathcal{L}_{\log}(H)$;
  \item or
  \begin{align*}
    T(x) = u x^{t} v
  \end{align*}
  for all $x \in\mathcal{L}_{\log}(H)$,
  where $x^{t}$ denotes the transpose of $x$ with respect to a fixed complete orthonormal basis of $H$.
\end{enumerate}
\end{corollary}

The topologies generated by  $\left\|{\cdot}\right\|_{\log}$ and $\left\|{\cdot}\right\|_{1}$ coincide (see Proposition \ref{coincide}) and  
the spaces of surjective
isometries on $\mathcal{L}_{\log}(H)$ and $L_1(H)$ coincide (see Corollary~\ref{cor:log-iso} and \cite{Arazy,Erdos}).
By 
  the main result of  \cite{BFJ} (see also \cite[Proposition 2]{Sourour} or \cite{BFGJ}), we obtain the following characterization of one-parameter group of surjective isometries on $\mathcal{L}_{\log}(H)$. 
\begin{theorem}
 The set  $\{T_t\}_{t\in \mathbb{R}}$ is  a strongly continuous group of surjective isometries on $\mathcal{L}_{\log}(H)$ if and only if there exist self-adjoint (not necessarily bounded) operators $a$ and $b$ in $B(H)$ such that 
 $$T_t (x) = e^{it a } x e^{itb}, ~x\in\mathcal{L}_{\log}(H).$$
 If $\alpha$ is the infinitesimal generator of $\{T_t\}_{t\in \mathbb{R}}$, then 
 $$\alpha(x) = i(ax+xb)$$
 for all $x$ in the domain of $\alpha$. 
 The domain of $\alpha$ is precisely the set of all operators in $\mathcal{L}_{\log}(H)$ which map the domain of $b$ into the domain of $a$ and for which the closure of $ax+xb$ belongs to $\mathcal{L}_{\log}(H)$.

 The group $\{T_t\}_{t\in \mathbb{R}}$  is uniformly continuous if and only if $a$ and $b$ are bounded. 
\end{theorem}

\section*{Acknowledgements}
J. Huang, K. Kudaybergenov and B. Yan were supported the NNSF of China (No.12031004, 12301160 and 12471134) and by Basic Research Program of Jiangsu (BK20251783). 
F. Sukochev was  supported by the  ARC (DP230100434).

{\bf Conflict of interest:} On behalf of all authors, the corresponding author states
that there is no conflict of interest.

{\bf Data availability:} Not applicable.

\end{document}